\documentclass{amsart}

\usepackage{ amsthm }
\usepackage{ amssymb }
\usepackage{ amsmath }
\usepackage{ array }
\usepackage{ cite }
\usepackage{ color }
\usepackage{ enumitem }
\usepackage{ hyperref }
\usepackage{ latexsym }
\usepackage{ url }
\usepackage[all]{ xy }
\theoremstyle{definition}

\newtheorem{definition}{Definition}[section]

\newtheorem{example}[definition]{Example}

\newtheorem{question}[definition]{Question}
\newtheorem{remark}[definition]{Remark}

\theoremstyle{plain}

\newtheorem{lemma}[definition]{Lemma}
\newtheorem{proposition}[definition]{Proposition}
\newtheorem{theorem}[definition]{Theorem}
\newtheorem{corollary}[definition]{Corollary}

\allowdisplaybreaks

\begin{document}

\title{Poisson triple systems}

\author{Murray R. Bremner}

\address{Department of Mathematics and Statistics,
University of Saskatchewan,
Saskatoon, Canada}

\email{bremner@math.usask.ca}

\author{Hader A. Elgendy}

\address{Department of Mathematics, Faculty of Science, Damietta University, New Damietta, Egypt}

\email{haderelgendy42@hotmail.com}

\subjclass[2010]{Primary 17B63. Secondary 17A30, 17A40, 17B35, 18M70, 68W30}

\keywords{Poisson algebras, Poisson triple systems, universal enveloping algebras,
algebraic operads, Koszul operads, computer algebra}

\thanks{The first author was supported by the Discovery Grant \emph{Algebraic Operads} from NSERC,
the Natural Sciences and Engineering Research Council of Canada.
The second author was supported by a Postdoctoral Fellowship from
the Ministry of Higher Education of Egypt.}

\begin{abstract}
We introduce Poisson triple systems, 
which are vector spaces with 3 trilinear operations 
satisfying 9 polynomial identities of degree 5.
We show that every Poisson triple system has a universal enveloping Poisson algebra.   
Finally, we briefly discuss operadic aspects of Poisson triple systems.
\end{abstract}

\maketitle

%{\footnotesize\tableofcontents}

%%%%%%%%%%%%%%%%%%%%%%%%%%%%%%%%%%%%%%%%%%%%%%%%%%%%%%%%%%%%%%%%%%%%%%%%%%%%%%%%%%%%%%%%%%%%%%%%
%%%%%%%%%%%%%%%%%%%%%%%%%%%%%%%%%%%%%%%%%%%%%%%%%%%%%%%%%%%%%%%%%%%%%%%%%%%%%%%%%%%%%%%%%%%%%%%%
%%%%%%%%%%%%%%%%%%%%%%%%%%%%%%%%%%%%%%%%%%%%%%%%%%%%%%%%%%%%%%%%%%%%%%%%%%%%%%%%%%%%%%%%%%%%%%%%

\section{Introduction}

Poisson algebras have two binary operations: 
a commutative associative operation, denoted $a \cdot b$,
and an anticommutative operation satisfying the Jacobi identity, denoted $[a,b]$.
The bracket acts by derivations of the commutative product.

\begin{definition}
\label{classical}
A \emph{Poisson algebra} is a vector space with bilinear operations $a \cdot b$ and $[a,b]$ 
satisfying these relations:
\begin{alignat*}{2}
& 1) &\quad
& a \cdot b = b \cdot a
\\
& 2) &\quad
& [a,b] + [b,a] = 0
\\
& 3) &\quad
&(a \cdot b) \cdot c = a \cdot (b \cdot c)
\\
& 4) &\quad
&[[a,b],c] + [[b,c],a] + [[c,a],b] = 0
\\
& 5) &\quad
&[a,b \cdot c] = [a,b] \cdot c + b \cdot [a,c]
\end{alignat*}
\end{definition}

Markl \& Remm \cite{MR} developed the notion of polarization and depolarization of
binary operations, which leads to an equivalent definition of Poisson algebras
in terms of a single binary operation (with no symmetry) satisfying a single relation.

\begin{definition}
Given a bilinear operation $ab$, the corresponding \emph{polarized} operations
are defined by
\[
a \cdot b = \tfrac12 ( ab + ba ),
\qquad
[a,b] = \tfrac12 ( ab - ba ).
\]
Conversely, given two bilinear operations,
one commutative $a \cdot b$ and one anticommutative $[a,b]$,
the corresponding \emph{depolarized} operation
is defined by
\[
ab = a \cdot b + [a,b].
\]
\end{definition}

\begin{definition}
\label{oneoperation}
A \emph{Poisson algebra} is a vector space with
a bilinear operation $ab$ satisfying this relation:
\[
a ( b c )
=
( a b ) c
-
\tfrac13
\big\{
(a c) b - (b a) c + (b c) a - (c a) b
\big\}.
\]
\end{definition}

\begin{lemma}
The two definitions of Poisson algebra are equivalent.
\end{lemma}

\begin{proof}
Markl \& Remm \cite[Example 2]{MR}.
\end{proof}

\begin{example}\label{Ex0}
Every Lie algebra $L$ is a Poisson algebra with respect to the zero associative product: 
$a \cdot  b = 0$.
\end{example}

\begin{example}\label{Ex1}
The commutative associative polynomial ring $\mathbb{F}[X, Y]$
becomes a Poisson algebra by defining the bracket as follows:
\[
[f,g] 
= 
\frac{\partial f} {\partial X}  \frac{\partial g} {\partial Y}  
-    
\frac{\partial f } {\partial Y}  \frac{\partial g} {\partial X}.
\]
\end{example}

\begin{example}
Let $M$ be a smooth manifold and let $C^{\infty}$ denote the real algebra of 
smooth real-valued functions on $M$, where the multiplication is defined pointwise. 
We call $M$ a Poisson manifold if there is a bilinear map   
$[-,-]\colon C^{\infty}(M) \times C^{\infty}(M) \to C^{\infty}(M)$ 
which defines a structure of Poisson algebra on $C^{\infty} (M) $.
\end{example}

For a basic reference on Poisson structures, see \cite{Laurent-Gengoux}.
Poisson algebras originated in the study of Hamiltonian mechanics \cite{Poisson}
and Poisson structures on smooth manifolds \cite{Lichnerowicz,Marsden}. 
For more applications of Poisson algebras in physics and geometry, 
see \cite{Crainic1,Fernandes,KOSMANN,Weinstein}.
Poisson algebras have also been studied from the points of view of 
nonassociative algebra \cite{GR2008,MLS2012,MLU2007,Shestakov,ZCB} and
algebraic operads \cite{Chapoton,Dotsenko-2007,Fresse}. 

A triple system is a vector space $T$ together with a trilinear map $T^3 \to T$.  
One may mention associative triple systems \cite{Lister}, Lie triple systems \cite{ListerLie}, 
Jordan triple systems \cite{Neher}, and Leibniz triple systems \cite{BSO}.
Triple systems are rich in algebraic structures, 
and they provide important common ground for various branches of mathematics, 
not only pure algebra and differential geometry, 
but also representation theory and algebraic geometry. 

This paper is organized as follows.
Section \ref{PTS} introduces Poisson triple systems (PTS),
which are related to Poisson algebras in the same way that associative triple systems
are related to associative algebras. 
We define a PTS to be a vector space with three trilinear operations satisfying 9 polynomial identities of degree 5 (Definition \ref{PTSdefinition}).
Section \ref{computational} discusses the computational methods that we used
to determine the defining identities for PTS.
Section \ref{enveloping} shows that every PTS has
a universal enveloping Poisson algebra, and that the enveloping algebra is finite-dimensional
if the original PTS is finite-dimensional. 
In Section \ref{operads} we use results of Dotsenko, Markl and Remm \cite{DMR}
to related PTS to the notion of Veronese powers of operads.
Throughout we work over a field $\mathbb{F}$ of characteristic 0.

%%%%%%%%%%%%%%%%%%%%%%%%%%%%%%%%%%%%%%%%%%%%%%%%%%%%%%%%%%%%%%%%%%%%%%%%%%%%%%%%%%%%%%%%%%%%%%%%
%%%%%%%%%%%%%%%%%%%%%%%%%%%%%%%%%%%%%%%%%%%%%%%%%%%%%%%%%%%%%%%%%%%%%%%%%%%%%%%%%%%%%%%%%%%%%%%%
%%%%%%%%%%%%%%%%%%%%%%%%%%%%%%%%%%%%%%%%%%%%%%%%%%%%%%%%%%%%%%%%%%%%%%%%%%%%%%%%%%%%%%%%%%%%%%%%

\section{Poisson triple systems}
\label{PTS}

\begin{definition}
\label{ternarydefinition}
In any Poisson algebra we consider three ternary operations
defined in terms of the binary operations as follows:
\[
\langle a, b, c \rangle = a \cdot b \cdot c,
\qquad
( a, b, c ) = [ a, b ] \cdot c,
\qquad
[ a, b, c ] = [ [ a, b ], c ].
\]
\end{definition}

\begin{remark}
We do not consider $[ a \cdot b, c ]$ since it can be expressed using $( a, b, c )$:
\[
[ a \cdot b, c ]
=
[ a, c ] \cdot b + a \cdot [ b, c ]
=
[ a, c ] \cdot b + [ b, c ] \cdot a
=
( a, c, b ) + ( b, c, a ).
\]
\end{remark}

\begin{lemma}
\label{degree3symmetries}
Every symmetry (polynomial identity in degree 3) satisfied by the operations
of Definition \ref{ternarydefinition} is a consequence of the following relations:
\begin{alignat*}{2}
&\langle a, b, c \rangle = \langle a^\sigma, b^\sigma, c^\sigma \rangle \quad ( \sigma \in S_3 ),
&\quad\quad\quad
&(a,b,c) + (b,a,c) = 0,
\\
&[a,b,c] + [b,a,c] = 0,
&\quad\quad\quad
&[a,b,c] + [b,c,a] + [c,a,b] = 0.
\end{alignat*}
\end{lemma}

\begin{proof}
This follows immediately from Definition \ref{ternarydefinition}.
\end{proof}

\begin{definition}
\label{PTSdefinition}
A \emph{Poisson triple system} (PTS) is a vector space $V$ with
three trilinear maps $f, g, h \colon V^3 \to V$ written using the notation
\[
f(a,b,c) = \langle a,b,c \rangle, \qquad
g(a,b,c) = (a,b,c), \qquad
h(a,b,c) = [a,b,c],
\]
which satisfy the symmetries of Lemma \ref{degree3symmetries} and the following relations:
\begin{alignat*}{2}
& 1) & \quad
% PTS defining identity 1
%
\langle \langle a,b,c \rangle, d,e \rangle
&=
\langle a,b, \langle c,d,e \rangle \rangle
\\
& 2) & \quad
% PTS defining identity 2
%
  \langle (a,b,c),d,e \rangle
  &=
  (a,b, \langle c,d,e \rangle)
\\
& 3) & \quad
% PTS defining identity 3
%
  (a,b,(c,d,e))
  &=
  (c,d,(a,b,e))
\\
& 4) & \quad
% PTS defining identity 4
%
  \langle [a,b,c],d,e \rangle
&=
  ((a,b,e),c,d)
+ (a,b,(c,e,d))
\\
& 5) & \quad
% PTS defining identity 5
%
  (\langle a,b,c \rangle, d,e)
  &=
  (a,d, \langle b,c,e \rangle)
+ (b,d, \langle a,c,e \rangle)
+ (c,d, \langle a,b,e \rangle)
\\
& 6) & \quad
% PTS defining identity 6
%
  [(a,b,c),d,e]
&=
  ([a,b,d],e,c)
- (a,b,[d,c,e])
- (d,c,[a,b,e])
- (e,c,[a,b,d])
\\
& 7) & \quad
% PTS defining identity 7
%
  ([a,b,c],d,e)
&=
  ([a,d,b],c,e)
- ([b,d,a],c,e)
- ([c,d,a],b,e)
+ ([c,d,b],a,e)
\\
& 8) & \quad
% PTS defining identity 8: ternary derivation for Lie triple systems
%
[a,b,[c,d,e]]
&=
[[a,b,c],d,e]
+
[c,[a,b,d],e]
+
[c,d,[a,b,e]]
\\
& 9) & \quad
% PTS defining identity 9
%
[ \langle a,b,c \rangle, d,e ]
&=
  ((a,d,b),e,c)
- (a,d,(e,c,b))
+ ((b,d,a),e,c)
\\
&&&\qquad {}
- (b,d,(e,c,a))
+ ((c,d,b),e,a)
- (c,d,(e,a,b))
\end{alignat*}
\end{definition}

\begin{remark}
Section \ref{computational} shows how we found the relations in Definition \ref{PTSdefinition}.
\end{remark}

\begin{theorem}
\label{relationssatisfied}
The relations in Definition \ref{PTSdefinition} are satisfied by the ternary operations
of Definition \ref{ternarydefinition} in every Poisson algebra.
\end{theorem}

\begin{proof}
Relation 1 is associativity for the commutative triple product.
Relation 8 is the derivation property of Lie triple systems.
The other relations can be proved by straightforward expansion.
For example, the left side relation 9 is
\[
\text{LS}
=
[ \langle a,b,c \rangle, d,e ]
=
[ \, [ a \cdot b \cdot c, d ], e ]
\]
Since $[-,d]$ is a derivation, and using commutativity, we can rewrite this as
\[
\text{LS}
=
[ \, [ a, d ] \cdot b \cdot c, e ]
+
[ \, [ b, d ] \cdot a \cdot c, e ]
+
[ \, [ c, d ] \cdot a \cdot b, e ]
\]
Since $[-,e]$ is a derivation, and again using commutativity, we obtain
\begin{align*}
\text{LS}
&=
[ [ a, d ], e ] \cdot b \cdot c
+
[ a, d ] \cdot [ b, e ] \cdot c
+
[ a, d ] \cdot [ c, e ] \cdot b
\\
&\quad {}
+
[ [ b, d ], e ] \cdot a \cdot c
+
[ a, e ] \cdot [ b, d ] \cdot c
+
[ b, d ] \cdot [ c, e ] \cdot a
\\
&\quad {}
+
[ [ c, d ], e ] \cdot a \cdot b
+
[ a, e ] \cdot [ c, d ] \cdot b
+
[ b, e ] \cdot [ c, d ] \cdot a
\end{align*}
The right side of relation 9 is
\begin{align*}
\text{RS}
&=
  [ [a,d] \cdot b, e ] \cdot c
+ [a,d] \cdot [c,e] \cdot b
+ [ [b,d] \cdot a, e ] \cdot c
\\
&\quad {}
+ [b,d] \cdot [c,e] \cdot a
+ [ [c,d] \cdot b, e ] \cdot a
+ [a,e] \cdot [c,d] \cdot b
\end{align*}
Expanding terms 1, 3, 5 gives
\begin{align*}
[ [a,d] \cdot b, e ] \cdot c &= [[a,d],e] \cdot b \cdot c + [a,d] \cdot [b,e] \cdot c,
\\
[ [b,d] \cdot a, e ] \cdot c &= [[b,d],e] \cdot a \cdot c + [a,e] \cdot [b,d] \cdot c,
\\
[ [c,d] \cdot b, e ] \cdot a &= [[c,d],e] \cdot a \cdot b + [b,e] \cdot [c,d] \cdot a.
\end{align*}
From this we see that the right and left sides are equal.
\end{proof}

\begin{theorem}
Every relation in degree 5 satisfied by the ternary operations
in Definition \ref{ternarydefinition} is a consequence of the relations
in Definition \ref{PTSdefinition}.
\end{theorem}

\begin{proof}
By computer algebra as explained in Section \ref{computational}.
\end{proof}

\begin{example}
The Poisson algebra of Example \ref{Ex0} gives rise to the following PTS:
\[
\langle a, b, c \rangle = 0 = ( a, b, c ) , \qquad 
[ a, b, c ] = [ [ a, b ], c ].
\]
\end{example}

\begin{example}
The Poisson algebra of Example \ref{Ex1} gives rise to the following PTS:
 \[\langle f,g,h \rangle = fgh,\]
\[  (f, g, h) = \frac{\partial f} {\partial X}  \frac{\partial g} {\partial Y} h  -    \frac{\partial f } {\partial Y}  \frac{\partial g} {\partial X}h,\]
 \[ [f,g,h]=  \frac{\partial}{\partial X}\left(  \frac{\partial f}{\partial X}\frac{\partial g}{\partial Y}
 - \frac{\partial f}{\partial Y}  \frac{\partial g}{\partial X}\right) \frac{\partial h}{\partial Y}
 -  \frac{\partial}{\partial Y}\left(  \frac{\partial f}{\partial X}\frac{\partial g}{\partial Y}
 - \frac{\partial f}{\partial Y}  \frac{\partial g}{\partial X}\right) \frac{\partial h}{\partial X}. \]
\end{example}

%%%%%%%%%%%%%%%%%%%%%%%%%%%%%%%%%%%%%%%%%%%%%%%%%%%%%%%%%%%%%%%%%%%%%%%%%%%%%%%%%%%%%
%%%%%%%%%%%%%%%%%%%%%%%%%%%%%%%%%%%%%%%%%%%%%%%%%%%%%%%%%%%%%%%%%%%%%%%%%%%%%%%%%%%%%
%%%%%%%%%%%%%%%%%%%%%%%%%%%%%%%%%%%%%%%%%%%%%%%%%%%%%%%%%%%%%%%%%%%%%%%%%%%%%%%%%%%%%

\section{Computational methods}
\label{computational}

First, we consider free algebras with three abstract ternary operations denoted  
$\langle -, -, - \rangle$, $( -, -, - )$, $[ -, -, - ]$
satisfying the symmetries of Lemma \ref{degree3symmetries}.

\begin{definition}
An \emph{association type} in degree $n$ is a valid placement of ternary operation symbols
into the sequence $a_1 \cdots a_n$.
A \emph{multilinear monomial} is obtained by permuting the subscripts of the arguments.
\end{definition}

\begin{lemma}
In degree 5, it suffices to consider the association types in Table \ref{ternarydegree5}.
If we apply a permutation $\sigma$ to the arguments $a, b, c, d, e$
then each association type has the indicated symmetries 
where $\prec$ denotes lexicographical order.
It follows that each association type has the indicated number 
of distinct multilinear monomials.
\end{lemma}

\begin{proof}
The claims follow immediately from Lemma \ref{degree3symmetries}.
For example,
\[
[c,a,b] = - [a,c,b],
\qquad
[b,c,a] = - [a,b,c] - [c,a,b] =  - [a,b,c] + [a,c,b],
\]
and hence in any occurrence of $[ -, -, - ]$ the lexicographically
first argument can be moved to the first position.
\end{proof}

\begin{definition}
We write $T(5)$ for the vector space whose basis consists of the 360 multilinear monomials 
in Table \ref{ternarydegree5}.
\end{definition}

\begin{table}[ht]
  \begin{alignat*}{4}
  & & &\text{association type} & & \text{symmetries} & & \text{monomials}
  \\
  1 &\qquad & &\langle \langle a, b, c \rangle, d, e \rangle  &\qquad
  & a^\sigma \prec b^\sigma \prec c^\sigma, \; d^\sigma \prec e^\sigma &\qquad
  & 10
  \\
  2 &\qquad & &\langle (a, b, c), d, e \rangle  &\qquad
  & a^\sigma \prec b^\sigma, \; d^\sigma \prec e^\sigma &\qquad
  & 30
  \\
  3 &\qquad & &\langle [a, b, c], d, e \rangle  &\qquad
  & a^\sigma \prec b^\sigma, \; a^\sigma \prec c^\sigma, \; d^\sigma \prec e^\sigma &\qquad
  & 20
  \\
  4 &\qquad & &( \langle a, b, c \rangle, d, e ) &\qquad
  & a^\sigma \prec b^\sigma \prec c^\sigma &\qquad
  & 20
  \\
  5 &\qquad & &( a, b, \langle c, d, e \rangle ) &\qquad
  & a^\sigma \prec b^\sigma, \; c^\sigma \prec d^\sigma \prec e^\sigma &\qquad
  & 10
  \\
  6 &\qquad & &( ( a, b, c ), d, e ) &\qquad
  & a^\sigma \prec b^\sigma &\qquad
  & 60
  \\
  7 &\qquad & &( a, b, ( c, d, e ) ) &\qquad
  & a^\sigma \prec b^\sigma, \; c^\sigma \prec d^\sigma &\qquad
  & 30
  \\
  8 &\qquad & &( [ a, b, c ], d, e ) &\qquad
  & a^\sigma \prec b^\sigma, \; a^\sigma \prec c^\sigma &\qquad
  & 40
  \\
  9 &\qquad & &( a, b, [ c, d, e ] ) &\qquad
  & a^\sigma \prec b^\sigma, \; c^\sigma \prec d^\sigma, \; c^\sigma \prec e^\sigma &\qquad
  & 20
  \\
  10 &\qquad & &[\langle a, b, c \rangle, d, e ] &\qquad
  & a^\sigma \prec b^\sigma \prec c^\sigma &\qquad
  & 20
  \\
  11 &\qquad & &[ ( a, b, c), d, e ] &\qquad
  & a^\sigma \prec b^\sigma &\qquad
  & 60
  \\
  12 &\qquad & &[ [ a, b, c], d, e ] &\qquad
  & a^\sigma \prec b^\sigma, \; a^\sigma \prec c^\sigma &\qquad
  & 40
  \end{alignat*}
\caption{Association types, symmetries, and numbers of 
multilinear monomials in degree 5 for the three ternary operations}
\label{ternarydegree5}
\end{table}

Second, we consider a monomial basis for the multilinear subspaces 
of free Poisson algebras (which can be identified with
symmetric algebras of free Lie algebras).

\begin{definition}
If $X = \{ a_1, \dots, a_n \}$ is a set of $n$ generators, 
then $P(X)$ denotes the free Poisson algebra generated by $X$,
$L(X)$ the free Lie algebra on $X$, and
$S(X)$ the free symmetric (commutative associative) algebra on $X$.
We write $P(n)$ for the multilinear subspace of degree $n \ge 1$ 
in the free Poisson algebra on $n$ generators,
and similarly for the other cases.
\end{definition}

\begin{lemma}
\label{freePoissonlemma}
As $\mathbb{Z}$-graded vector spaces we have $P(X) \cong S(L(X))$.
(Strictly speaking, we should write $S(Y)$ where $Y$ is a basis of $L(X)$.)
\end{lemma}

\begin{proof}
See \cite{Mishchenko} and \cite[Lemma 1]{Shestakov}.
\end{proof}
 
\begin{lemma}
We have $\dim L(n) = (n{-}1)!$ and a basis consists of all left-normed Lie monomials
with the lexicographically first argument in the first position:
\[
[ \cdots [ [ a_{\sigma(1)}, a_{\sigma(2)} ], a_{\sigma(3)} ], \dots, a_{\sigma(n)} ],  
\quad
\sigma \in S_n, \quad \sigma(1)=1.
\]
\end{lemma}

\begin{proof}
Reutenauer \cite[\S5.6.2]{Reutenauer}.
\end{proof}

\begin{corollary}
Table \ref{Poissondegree5} gives a basis for $P(5)$.
\end{corollary}

\begin{table}[ht]
  \begin{alignat*}{5}
  & &\quad
  &\text{partition} &\quad
  &\text{association type} &\quad
  &\text{symmetries} &\quad
  &\text{monomials}
  \\
  & 1 &\quad
  & 5 & \quad
  & [[[[a,b],c],d],e] &\quad
  & a^\sigma \prec b^\sigma, c^\sigma, d^\sigma, e^\sigma &\quad
  & 24
  \\
  & 2 &\quad
  & 41 &\quad
  & [[[a,b],c],d] \cdot e &\quad
  & a^\sigma \prec b^\sigma, c^\sigma, d^\sigma &\quad
  & 30
  \\
  & 3 &\quad
  & 32 &\quad
  & [[a,b],c] \cdot [d,e] &\quad
  & a^\sigma \prec b^\sigma, c^\sigma; \; d^\sigma \prec e^\sigma &\quad
  & 20
  \\
  & 4 &\quad
  & 311 &\quad
  & [[a,b],c] \cdot d \cdot e &\quad
  & a^\sigma \prec b^\sigma, c^\sigma; \; d^\sigma \prec e^\sigma&\quad
  & 20
  \\
  & 5 &\quad
  & 221 &\quad
  & [a,b] \cdot [c,d] \cdot e &\quad
  & a^\sigma \prec b^\sigma; \; c^\sigma \prec d^\sigma; \; a^\sigma \prec c^\sigma &\quad
  & 15
  \\
  & 6 &\quad
  & 2111 &\quad
  & [a,b] \cdot c \cdot d \cdot e &\quad
  & a^\sigma \prec b^\sigma; \; c^\sigma \prec d^\sigma \prec e^\sigma &\quad
  & 10
  \\
  & 7 &\quad
  & 11111 &\quad
  & a \cdot b \cdot c \cdot d \cdot e &\quad
  & a^\sigma \prec b^\sigma \prec c^\sigma \prec d^\sigma \prec e^\sigma &\quad
  & 1
  \end{alignat*}
\caption{Association types, symmetries, and numbers
of multilinear monomials of degree 5 in the free Poisson algebra}
\label{Poissondegree5}
\end{table}

\begin{definition}
The \emph{expansion map} is a linear map $R\colon T(5) \to P(5)$ defined on basis monomials 
from Table \ref{ternarydegree5}.
To compute $R$, each ternary operation in the monomial basis of $T(5)$ is expanded 
into a binary Poisson monomial by applying Definition \ref{ternarydefinition}.
The result will be a linear combination of Poisson monomials
which are not necessarily in the form of Table \ref{Poissondegree5}.
These Poisson monomials then require further straightening using the Jacobi identity.
\end{definition}

\begin{example}
In degree 5, the only non-trivial cases are association types 4, 6, 10, 11 in Table \ref{ternarydegree5}.
For example, the expansion of type 11 is as follows:
\begin{align*}
[(a,b,c),d,e]
&=
[ \, [ \, [ a, b ] \cdot c, d ], e ]
\\
&=
[ \, [ \, [ a, b ], d ] \cdot c, e ]
+
[ \, [ a, b ] \cdot [ c, d ], e ]
\\
&=
[ \, [ \, [ a, b ], d ], e ] \cdot c
+
[ \, [ a, b ], d ] \cdot [ c, e ]
+
[ \, [ a, b ], e ] \cdot [ c, d ]
+
[ a, b ] \cdot [ \, [ c, d ], e ]
\\
&=
[ \, [ \, [ a, b ], d ], e ] \cdot c
+
[ \, [ a, b ], d ] \cdot [ c, e ]
+
[ \, [ a, b ], e ] \cdot [ c, d ]
+
[ \, [ c, d ], e ] \cdot [ a, b ]
\end{align*}
\end{example}

\begin{remark}
The kernel of $R$ consists of the relations satisfied by the ternary operations of
Definition \ref{ternarydefinition} in every Poisson algebra.
\end{remark}

\begin{definition}
The $120 \times 360$ \emph{expansion matrix} $M$ represents
$R$ with respect to the ordered monomial bases 
in Tables \ref{ternarydegree5} and \ref{Poissondegree5}:
the $(i,j)$ entry of $M$ is the coefficient of the $i$-th Poisson monomial
in the expansion of the $j$-th ternary monomial.
\end{definition}

We use computer algebra (Maple) to find a set of
$S_5$-module generators for the nullspace of $M$.
We first determine that $M$ has full rank 120 and nullity 240.
We compute the Hermite normal form $H$ of the transpose $M^t$ 
together with a square integer matrix $U$ of size 360 satisfying $U M^t = H$.
We use the Maple command \texttt{LinearAlgebra[HermiteForm]}
with the option \texttt{method='integer[reduced]'} so that the LLL algorithm
is applied to reduce the size of the entries of $U$
all of which belong to $\{ 0, \pm 1 \}$.
The bottom 240 rows of $U$ form a matrix $N$ whose rows form a basis 
for the left nullspace of $M^t$, which is the right nullspace of $M$.
The rows of $N$ are the coefficient vectors of the relations satisfied by 
the ternary operations of Definition \ref{ternarydefinition} in every Poisson algebra.

The size of a coefficient vector is the sum of the squares of its entries.
The size of a basis is the base-10 logarithm of the product of the sizes of the basis vectors.
We find that the basis of the nullspace consisting of the rows of $N$ has size $\approx 183.5$.
The shortest basis vector has size 2 and the longest has size 20.
We apply the LLL algorithm with increasing reduction parameters $3/4$, $9/10$, and $99/100$
to obtain bases of the nullspace with sizes 170.2, 165.2, and 143.3.
We sort the basis vectors by increasing size so that the relations with fewer terms
come first. 
The shortest basis vector still has size 2 but the longest now has size 10.

From this linear basis for the nullspace we extract
a set of $S_5$-module generators for the nullspace.
For $1 \le i \le 240$, we apply all permutations to the relation corresponding
to basis vector $i$, and use this to compare the module $G_{i-1}$ generated by
the first $i{-}1$ vectors to the module $G_i$ generated by the first $i$ vectors.
In most cases we find that $G_i / G_{i-1}$ is zero:
relation $i$ is contained in the module generated by relations $1, \dots, i{-}1$.
We find that we need only 9 of the 240 basis vectors to generate the entire
nullspace as an $S_5$-module.
These basis vectors are the coefficient vectors of the relations in Definition \ref{PTSdefinition}.

%%%%%%%%%%%%%%%%%%%%%%%%%%%%%%%%%%%%%%%%%%%%%%%%%%%%%%%%%%%%%%%%%%%%%%%%%%%%%%%%%%%%%
%%%%%%%%%%%%%%%%%%%%%%%%%%%%%%%%%%%%%%%%%%%%%%%%%%%%%%%%%%%%%%%%%%%%%%%%%%%%%%%%%%%%%
%%%%%%%%%%%%%%%%%%%%%%%%%%%%%%%%%%%%%%%%%%%%%%%%%%%%%%%%%%%%%%%%%%%%%%%%%%%%%%%%%%%%%

\section{Universal enveloping Poisson algebras}
\label{enveloping}

In this section we show that every Poisson triple system $T$ has
a universal Poisson enveloping algebra $U(T)$ and that the natural map
$T \to U(T)$ is injective.

Let $P(T)$ be the free Poisson algebra generated by $T$
with the binary operations denoted as usual by $a \cdot b$ and $[a,b]$.
Thus $P(T)$ has a natural $\mathbb{Z}$-grading for which $T$ is the homogeneous
component of degree 1; we write $\iota\colon T \hookrightarrow P(T)$ for the inclusion.
The ternary operations in $T$ carry over in the obvious way to $\iota(T)$.
It is known that $P(T)$ is linearly isomorphic as a $\mathbb{Z}$-graded vector space
to the non-unital tensor algebra on $T$; see for example \cite{BD-PORO}.

\begin{definition}
\label{idealdefinition}
Let $I(T) \subseteq P(T)$ be the ideal generated by all elements of the following forms
where $a, b, c \in T$:
\begin{align*}
& \iota(a) \cdot \iota(b) \cdot \iota(c) - \iota( \langle a, b, c \rangle ),
\\
& [ \iota(a), \iota(b) ] \cdot \iota(c) - \iota( ( a, b, c ) ),
\\
& [ [ \iota(a), \iota(b) ], \iota(c) ] - \iota( [ a, b, c ] ),
\end{align*}
In the first terms of the generators, the operations are in $P(T)$, giving monomials of degree 3.
The second terms are elements of degree 1 in $\iota(T)$.
\end{definition}

\begin{definition}
\label{universal}
The \emph{universal Poisson envelope} of $T$ is the quotient Poisson algebra $U(T) = P(T)/I(T)$.
For $a \in T$ we write $\overline{a} = \iota(a) + I(T) \in U(T)$.
Clearly $U(T)$ is generated by the elements $\overline{a}$ for $a \in T$.
\end{definition}

\begin{lemma}
\label{U(T)lemma}
(i)
The intersection of $I(T)$ with $\iota(T) \oplus ( \iota(T) \otimes \iota(T) ) \subset P(T)$ 
is zero.
(ii)
We have the linear isomorphism $U(T) \cong \iota(T) \oplus ( \iota(T) \otimes \iota(T) )$.
\end{lemma}

\begin{proof}
From Definition \ref{idealdefinition} it is clear that every element of $I(T)$ 
contains a term of degree $n \ge 3$, which proves ($i$).
Furthermore, in $U(T)$ we have (for all $a, b, c \in T$)
\[
\overline{a} \cdot \overline{b} \cdot \overline{c}
=
\overline{ \langle a, b, c \rangle },
\qquad
[ \overline{a}, \overline{b} ] \cdot \overline{c}
=
\overline{ ( a, b, c ) },
\qquad
[ [ \overline{a}, \overline{b} ], \overline{c} ]
=
\overline{ [ a, b, c ] }.
\]
Hence every product of three generators of $U(T)$ is again a generator.
By induction it follows that in $U(T)$
every product of $n \ge 3$ generators is either a generator (for $n$ odd)
or a product of two generators (for $n$ even), which proves ($ii$).
\end{proof}

Lemma \ref{U(T)lemma} shows that we may identify $U(T)$ with $T \oplus ( T \otimes T )$.

\begin{theorem}
On $U(T)$ the commutative associative product and the Poisson bracket are defined as follows;
for simplicity we omit the bars:
\begin{alignat*}{2}
&(i) 
&\quad
a \cdot b 
&= 
\tfrac12 ( 
a \otimes b
+
b \otimes a
)
\\
&(ii)
&\quad
[ a, b ] 
&= 
\tfrac12 ( 
a \otimes b 
-
b \otimes a
)
\\
&(iii)
&\quad
a 
\cdot 
( b \otimes c )
&=
\langle a, b, c \rangle
+
( b, c, a )
\\
&(iv)
&\quad
[
a,
b \otimes c
]
&=
( a, b, c )
+
( a, c, b )
-
[ b, c, a ]
\\
&(v)
&\quad
( a \otimes b )
\cdot
( c \otimes d )
&=
\tfrac12
\big(
\langle a, b, c \rangle \otimes d +
d \otimes \langle a, b, c \rangle +
( c, d, a ) \otimes b +
b \otimes ( c, d, a ) + {}
\\
&&&
\quad\quad\;
( a, b, c ) \otimes d +
d \otimes ( a, b, c ) +
[ a, b ] \otimes [ c, d ] +
[ c, d ] \otimes [ a, b ]
\big)
\\
&(vi)
&\qquad
[ 
a \otimes b,
c \otimes d
]
&=
\tfrac{1}{2} 
\big( 
( a, c, b ) \otimes d 
+ 
d \otimes ( a, c, b ) 
+ 
( b, c, a ) \otimes d  
+ 
d \otimes ( b, c, a )
\\ 
&&&\qquad 
+
( a, d, b ) \otimes c 
+ 
c \otimes ( a, d, b ) 
+ 
( b, d, a ) \otimes c   
+ 
c \otimes ( b, d, a )
\\
&&&\qquad
+
[ a, b, c ] \otimes d + d \otimes [ a, b, c ] 
+ 
[a , b, d ] \otimes c + c \otimes [a , b, d ]
\big)
\\
&&&\quad
- [ c, d, a ] \otimes b - a \otimes [ c, d, b ] 
\end{alignat*}
\end{theorem}

\begin{proof}
Equations ($i$) and ($ii$) are obvious and imply that
\[
a \otimes b 
=
a \cdot b 
+
[ a, b ].
\]
To prove equations ($iii$) and ($iv$) we calculate as follows:
\begin{align*}
a 
\cdot 
( b \otimes c )
&=
a 
\cdot 
(
b \cdot c 
+
[ b, c ]
)
=
a 
\cdot 
b \cdot c 
+
a 
\cdot 
[ b, c ]
=
a 
\cdot 
b \cdot c 
+
[ b, c ]
\cdot
a 
\\
&=
\langle a, b, c \rangle
+
( b, c, a ),
\\
[
a,
b \otimes c
]
&=
[
a,
b \cdot c 
+
[ b, c ]
]
=
[
a,
b \cdot c 
]
+
[
a, 
[ b, c ]
]
=
[ a, b ] \cdot c
+
b \cdot [ a, c ]
-
[
[ b, c ],
a
]
\\
&
=
[ a, b ] \cdot c
+
[ a, c ] \cdot b
-
[ [ b, c ], a ]
=
( a, b, c )
+
( a, c, b )
-
[ b, c, a ].
\end{align*}
For equation ($v$) we calculate as follows and then apply equation ($i$):
\begin{align*}
( a \otimes b )
\cdot
( c \otimes d )
&=
( a \cdot b + [ a, b ] )
\cdot
( c \cdot d + [ c, d ] )
\\
&=
a \cdot b \cdot c \cdot d +
a \cdot b \cdot [ c, d ] +
[ a, b ] \cdot c \cdot d +
[ a, b ] \cdot [ c, d ]
\\
&=
a \cdot b \cdot c \cdot d +
[ c, d ] \cdot a \cdot b +
[ a, b ] \cdot c \cdot d +
[ a, b ] \cdot [ c, d ]
\\
&=
\langle a, b, c \rangle \cdot d +
( c, d, a ) \cdot b +
( a, b, c ) \cdot d +
[ a, b ] \cdot [ c, d ],
\end{align*}
Finally, for equation ($vi$) we first calculate as follows:
\begin{align}
\label{final}
[ a \otimes b,
c \otimes d
] &
 = 
 [ a \cdot b 
+
[ a, b ], c \cdot d 
+
[ c, d ]]
\\
&=  [ a \cdot b , c \cdot d ] + [ a \cdot b ,  [ c , d]]
+  [ [a , b ],  c \cdot d ] + [ [ a , b ],  [ c , d ]]. \notag
\end{align}
For the first term on the right side of \eqref{final} we obtain
\begin{align}
\label{one}
[ a \cdot b , c \cdot d ] 
&= 
[ a \cdot b, c] \cdot d + c \cdot [ a \cdot b, d]
\\
&= 
[a,c] \cdot b \cdot d + a \cdot [b,c] \cdot d
+
c \cdot [a,d] \cdot b + c \cdot a \cdot [b,d]
\notag
\\
&=
[a,c] \cdot b \cdot d + [b,c] \cdot a \cdot d
+
[a,d] \cdot b \cdot c + [b,d] \cdot a \cdot c
\notag
\\
&= 
( a, c, b ) \cdot d + ( b, c, a ) \cdot d + ( a, d, b ) \cdot c + ( b, d, a ) \cdot c  
\notag
\\
& = 
\tfrac{1}{2} 
\big( 
( a, c, b ) \otimes d 
+ 
d \otimes ( a, c, b ) 
+ 
( b, c, a ) \otimes d  
+ 
d \otimes ( b, c, a )
\notag
\\ 
&\qquad 
+
( a, d, b ) \otimes c 
+ 
c \otimes ( a, d, b ) 
+ 
( b, d, a ) \otimes c   
+ 
c \otimes ( b, d, a )
\big).
\notag
\end{align}
For the second term on the right side of \eqref{final} we obtain
\begin{align}
\label{two}
[ a \cdot b ,  [ c , d]] 
&=
[ a, [ c, d ] ] \cdot b + a \cdot [ b, [ c, d ] ]
=
{} - [ [ c, d ], a ] \cdot b - [ [ c, d ], b ] \cdot a 
\\
&= 
{} - [ c, d, a ] \cdot b - [ c, d, b ] \cdot a 
\notag
\\
&=
-\tfrac{1}{2} 
\big( 
[ c, d, a ] \otimes b + b \otimes [ c, d, a ] 
+ 
[ c, d, b ] \otimes a + a \otimes [ c, d, b ] 
\big).
\notag
\end{align}
From this, for the third term on the right side of \eqref{final} we obtain
\begin{align}\label{three}
[ [ a , b], c \cdot d  ] 
&= 
\tfrac{1}{2} 
\big( 
[ a, b, c ] \otimes d + d \otimes [ a, b, c ] 
+ 
[a , b, d ] \otimes c + c \otimes [a , b, d ]
\big).
\end{align}
For the fourth term on the right side of \eqref{final} we obtain
\begin{align}\label{four}
[ [ a, b ], [ c, d] ]
&=
[ [ a, [ c, d ] ], b ] + [ a, [ b, [ c, d ] ] ]
=
{} - [ [ c, d ], a ], b ] + [ [ c, d ], b ], a ]
\\
&=
{} - [ [ c, d, a ], b ] + [ [ c, d, b ], a ]
\notag
\\
&= 
-
\tfrac{1}{2} 
\big(   
[ c, d, a ] \otimes b - b \otimes [ c, d, a ]   
- 
[ c, d, b ] \otimes a + a \otimes [ c, d, b ] 
\big). 
\notag
\end{align}
Substituting the results of \eqref{one}--\eqref{four} in \eqref{final} gives equation ($vi$).
\end{proof}

\begin{theorem}
Every polynomial identity satisfied by the ternary operations
$\langle a, b, c \rangle$,  $( a, b, c )$, $[ a, b, c ]$ 
of Definition \ref{ternarydefinition}
in every Poisson algebra is a consequence of the defining identities 
for Poisson triple systems in Definition \ref{PTSdefinition}.
\end{theorem}

\begin{proof}
Suppose to the contrary that $p(a_1, \dots , a_n) \equiv 0$ is 
a (nonzero multilinear) polynomial identity 
in $n$ indeterminates which is satisfied by the ternary operations of 
Definition \ref{ternarydefinition} in every Poisson algebra
but which is not a consequence of the defining identities 
for Poisson triple systems in Definition \ref{PTSdefinition}. 
Hence $p$ is a nonzero element of the free Poisson triple system $T$ 
on $n$ generators $a_1, \dots, a_n$. 
Let $P(T)$ be the free Poisson algebra on the same $n$ generators regarded as 
a Poisson triple system according to Definition \ref{ternarydefinition}. 
Consider the morphism of Poisson triple systems $f\colon T \to P(T)$
defined in the obvious way:
\[
\langle a, b, c \rangle \mapsto a \cdot b \cdot c, 
\qquad 
( a, b, c ) \mapsto  [ a, b ] \cdot c, 
\qquad 
[a,b,c]\mapsto [ [ a, b ], c ].
\]  
Then by definition of polynomial identity we have $f(p) = 0$ where $p \ne 0$.  
Let $U(T)$ be the universal enveloping Poisson algebra of $T$ from Definition \ref{universal}.
Lemma \ref{U(T)lemma}($ii$) shows that there is an injective homomorphism of 
Poisson triple systems $j \colon T \to U(T)$. 
By the universal property of the free Poisson algebra $P(T)$, 
there is (unique) surjective homomorphism $\phi \colon P(T) \to U(T)$, 
which is the identity map on the generators $a_1, \dots, a_n$ 
and satisfies $\phi \circ f = j$. 
Since $p$ is a nonzero element of the kernel of $f$, 
we see that $f$ is not injective and hence $j$ cannot be injective.
This contradiction shows that such a polynomial identity $p$ cannot exist.
\end{proof}

%%%%%%%%%%%%%%%%%%%%%%%%%%%%%%%%%%%%%%%%%%%%%%%%%%%%%%%%%%%%%%%%%%%%%%%%%%%%%%%%%%%%%
%%%%%%%%%%%%%%%%%%%%%%%%%%%%%%%%%%%%%%%%%%%%%%%%%%%%%%%%%%%%%%%%%%%%%%%%%%%%%%%%%%%%%
%%%%%%%%%%%%%%%%%%%%%%%%%%%%%%%%%%%%%%%%%%%%%%%%%%%%%%%%%%%%%%%%%%%%%%%%%%%%%%%%%%%%%

\section{The operad for Poisson triple systems}
\label{operads}

\begin{definition}
We write $P(n)$ for the multilinear subspace of degree $n$ in
the free Poisson algebra on $n$ generators.
We call $n$ the \emph{arity}, and
$P(n)$ the \emph{homogeneous component} of arity $n$.
Using the bilinear operations of \emph{partial composition} \cite{BDAOAC,LV,MSS}
the following direct sum becomes the \emph{Poisson operad}:
\[
\mathbf{P} = \bigoplus_{n \ge 1} P(n)
\]
\end{definition}

\begin{definition}
\label{ntupledefinition}
For $n \ge 2$ we consider the (multilinear) monomials in arity $n$
with the identity permutation of the arguments using the isomorphism
of Lemma \ref{freePoissonlemma}.
(Table \ref{Poissondegree5} displays these monomials in arity 5.)
We obtain a set of multilinear operations of arity $n$ which generate 
a suboperad $\mathbf{N} \subseteq \mathbf{P}$.
These generating operations satisfy symmetries in arity $n$
together with relations in arities $k(n{-}1)+1$ for $k \ge 2$.
The suboperad $\mathbf{N}$ is the operad governing
\emph{Poisson $n$-tuple systems}.
\end{definition}

\begin{example}
For $n = 2$ we have the operations $a \cdot b$ and $[a,b]$ 
(with the obvious symmetries) and the relations they satisfy in arity 3
(which define Poisson algebras).
For $n = 3$ the generating operations appear in Definition \ref{ternarydefinition},
the symmetries appear in Lemma \ref{degree3symmetries}, 
and the relations in arity 5 appear in Definition \ref{PTSdefinition}.
\end{example}

\begin{definition}
(See \cite{DMR}.)
Let $\mathbf{O}$ be an operad generated by operations of arity 2.
For $d \ge 1$, the \emph{naive $d$-th Veronese power} of $\mathbf{O}$
is the suboperad of $\mathbf{O}$ consisting of the homogeneous components
of arities $kd+1$ for $k \ge 0$.
(For $d = 1$ the first Veronese power is just the operad $\mathbf{O}$ itself.)
The (non-naive) \emph{$d$-th Veronese power} of $\mathbf{O}$ is the suboperad
of $\mathbf{O}$ generated by the operations of arity $d+1$.
In general the naive and non-naive Veronese powers do not coincide
\cite[Prop.~3.2]{DMR}.
\end{definition}

\begin{remark}
According to Definition \ref{ntupledefinition} 
the operad governing Poisson $n$-tuple systems
is the non-naive $(n{-}1)$-st Veronese power of the Poisson operad.
\end{remark}

\begin{proposition}
For the Poisson operad, the naive and non-naive $d$-th Veronese powers coincide.
\end{proposition}

\begin{proof}
We use Definition \ref{oneoperation} of Poisson algebras (one nonassociative operation).
Since the Poisson operad is regular \cite{BD-PORO}, 
every homogeneous component has a basis consisting of the left-normed monomials 
with all permutations of the arguments.
Hence every homogeneous component is spanned by
the orbits of the left-comb products under the action of
the symmetric group.
Now apply \cite[Prop.~3.9]{DMR}.
\end{proof}

\begin{lemma}
\label{quadraticlemma}
The Poisson operad has a quadratic Gr\"obner basis.
\end{lemma}

\begin{proof}
See \cite[Theorem~2.1]{Dotsenko2020}.
\end{proof}

\begin{corollary}
The Poisson operad is a Koszul operad.
\end{corollary}

\begin{proof}
An operad with a quadratic Gr\"obner basis is Koszul \cite[Theorem 6.3.3.2]{BDAOAC}.
\end{proof}

\begin{question}
\label{koszulquestion}
Is the operad governing Poisson triple systems a Koszul operad?
At first glance, it looks as though a positive answer follows from 
\cite[Theorem~3.13]{DMR} since
($i$)
as a weight-graded operad the Poisson operad $\mathbf{P}$ is generated by elements 
of weight 1 (operations of arity 2), and
($ii$)
by Lemma \ref{quadraticlemma} we know that the ideal of relations of $\mathbf{P}$ has 
a quadratic Gr\"obner basis.
However, the remaining hypothesis of \cite[Theorem~3.13]{DMR} does not hold:
it is not true (in the case $d = 2$) that all weight $2d$ normal tree monomials 
with respect to that Gr\"obner basis belong to the $d$-th Veronese power of the free operad.
If this hypothesis were true then it would follow by \cite[Theorem~3.13]{DMR} that 
$\mathbf{P}^{[d]}$ admits a quadratic Gr\"obner basis and hence that $\mathbf{P}^{[d]}$ 
is Koszul.
Vladimir Dotsenko provided us with the following counter-example:
the element $( a ( [b c] [d e] ) )$ is normal of weight 4 but is not 
in the Veronese square of the free operad. 
So this question remains open.
\end{question}

\begin{question}
Livernet \& Loday showed that the Poisson operad
can be expressed as a limit of associative operads; see \cite{MR}. 
Can the operad for Poisson $n$-tuple systems be expressed as a limit of operads
for associative $n$-tuple systems \cite{Carlsson}?
\end{question}

\subsubsection*{Acknowledgement}

We thank Vladimir Dotsenko for providing us with the counter-example in
Question \ref{koszulquestion}.

%%%%%%%%%%%%%%%%%%%%%%%%%%%%%%%%%%%%%%%%%%%%%%%%%%%%%%%%%%%%%%%%%%%%%%%%%%%%%%%%%%%%%
%%%%%%%%%%%%%%%%%%%%%%%%%%%%%%%%%%%%%%%%%%%%%%%%%%%%%%%%%%%%%%%%%%%%%%%%%%%%%%%%%%%%%
%%%%%%%%%%%%%%%%%%%%%%%%%%%%%%%%%%%%%%%%%%%%%%%%%%%%%%%%%%%%%%%%%%%%%%%%%%%%%%%%%%%%%

%%%%%%%%%%%%%%%%%%%%%%%%%%%%%%%%%%%%%%%%%%%%%%%%%%%%%%%%%%%%%%%%%%%%%%%%%%%%%%%%%%%%%%%%%%%%%%%%
%%%%%%%%%%%%%%%%%%%%%%%%%%%%%%%%%%%%%%%%%%%%%%%%%%%%%%%%%%%%%%%%%%%%%%%%%%%%%%%%%%%%%%%%%%%%%%%%
%%%%%%%%%%%%%%%%%%%%%%%%%%%%%%%%%%%%%%%%%%%%%%%%%%%%%%%%%%%%%%%%%%%%%%%%%%%%%%%%%%%%%%%%%%%%%%%%


\begin{thebibliography}{99}

\bibitem{BDAOAC}
Bremner, Murray R.; Dotsenko, Vladimir: 
Algebraic operads. An algorithmic companion. 
CRC Press, Boca Raton, FL, 2016. 

\bibitem{BD-PORO}
Bremner, Murray; Dotsenko, Vladimir:
Classification of regular parametrized one-relation operads.
Canad. J. Math. 69 (2017), no. 5, 992--1035. 

\bibitem{BSO}
Bremner, Murray R.; S\'anchez-Ortega, Juana:
Leibniz triple systems.
Commun. Contemp. Math. 16 (2014), no. 1, 1350051, 19 pp.

\bibitem{Carlsson}
Carlsson, Renate:
$n$-ary algebras. 
Nagoya Math. J. 78 (1980), 45--56. 

\bibitem{Chapoton}
Chapoton, Fr\'ed\'eric:
On a Hopf operad containing the Poisson operad.
Algebr. Geom. Topol. 3 (2003), 1257--1273.

\bibitem{Crainic1}
Crainic, Marius; Fernandes, Rui Loja:
Integrability of Poisson brackets.
J. Differential Geom. 66 (2004), no. 1, 71--137. 

\bibitem{Dotsenko-2007}
Dotsenko, Vladimir:
An operadic approach to deformation quantization of compatible Poisson brackets. I. 
J. Gen. Lie Theory Appl. 1 (2007), no. 2, 107--115.

\bibitem{Dotsenko2020}
Dotsenko, Vladimir:
Word operads and admissible orderings.
Appl. Categ. Structures 28 (2020), no. 4, 595--600.

\bibitem{DMR}
Dotsenko, Vladimir; Markl, Martin; Remm, Elisabeth:
Veronese powers of operads and pure homotopy algebras.
Eur. J. Math. 6 (2020), no. 3, 829--863.

\bibitem{Fernandes} 
Fernandes, Rui Loja:
Connections in Poisson geometry. I. Holonomy and invariants.
J. Differential Geom. 54 (2000), no. 2, 303--365.

\bibitem{Fresse}
Fresse, Benoit:
Th\'eorie des op\'erades de Koszul et homologie des alg\`ebres de Poisson.
[Theory of Koszul operads and homology of Poisson algebras] 
Ann. Math. Blaise Pascal 13 (2006), no. 2, 237--312. 

\bibitem{GR2008}
Goze, Michel; Remm, Elisabeth:
Poisson algebras in terms of non-associative algebras.
J. Algebra 320 (2008), no. 1, 294--317.

\bibitem{KOSMANN}
Kosmann-Schwarzbach, Yvette:
From Poisson algebras to Gerstenhaber algebras. 
Ann. Inst. Fourier (Grenoble) 46 (1996), no. 5, 1243--1274.

\bibitem{Laurent-Gengoux}
Laurent-Gengoux, Camille; Pichereau, Anne; Vanhaecke:
Poisson structures. 
Grundlehren der mathematischen Wissenschaften 
[Fundamental Principles of Mathematical Sciences], 347. 
Springer, Heidelberg, 2013.

\bibitem{Lichnerowicz}
Lichnerowicz, Andr\'e:
Les vari\'et\'es de Poisson et leurs alg\`ebres de Lie associ\'ees.
J. Differential Geometry 12 (1977), no. 2, 253--300. 

\bibitem{ListerLie}
Lister, William G.
A structure theory of Lie triple systems. 
Trans. Amer. Math. Soc. 72 (1952), 217--242.

\bibitem{Lister}
Lister, W. G.:
Ternary rings. 
Trans. Amer. Math. Soc. 154 (1971), 37--55. 

\bibitem{LV}
Loday, Jean-Louis; Vallette, Bruno:
Algebraic operads. 
Grundlehren der mathematischen Wissenschaften 
[Fundamental Principles of Mathematical Sciences], 346. 
Springer, Heidelberg, 2012.

\bibitem{MLS2012}
Makar-Limanov, Leonid; Shestakov, Ivan:
Polynomial and Poisson dependence in free Poisson algebras and free Poisson fields.
J. Algebra 349 (2012), 372--379.

\bibitem{MLU2007}
Makar-Limanov, Leonid; Umirbaev, Ualbai:
Centralizers in free Poisson algebras.
Proc. Amer. Math. Soc. 135 (2007), no. 7, 1969--1975.

\bibitem{MR}
Markl, M.; Remm, E.:
Algebras with one operation including Poisson and other Lie-admissible algebras. 
J. Algebra 299 (2006), no. 1, 171--189.

\bibitem{MSS}
Markl, Martin; Shnider, Steve; Stasheff, Jim:
Operads in algebra, topology and physics.
Mathematical Surveys and Monographs, 96. 
American Mathematical Society, Providence, RI, 2002. 

\bibitem{Marsden} 
Marsden, Jerrold E.; Weinstein, Alan:
The Hamiltonian structure of the Maxwell-Vlasov equations. 
Phys. D 4 (1981/82), no. 3, 394--406. 

\bibitem{Mishchenko}
Mishchenko, S. P.; Petrogradsky, V. M.; Regev, A.:
Poisson PI algebras.
Trans. Amer. Math. Soc. 359 (2007), no. 10, 4669--4694.

\bibitem{Neher}
Neher, Erhard:
Jordan triple systems by the grid approach. 
Lecture Notes in Mathematics, 1280. 
Springer-Verlag, Berlin, 1987.

\bibitem{Poisson} 
Poisson, Sim\'eon Denis:
M\'emoire sur la variation des constantes arbitraires dans les questions de m\'ecanique.
Journal de l'\'Ecole polytechnique, 15e cahier, 8 (1809), 266--344.

\bibitem{Reutenauer}
Reutenauer, Christophe:
Free Lie algebras.
London Mathematical Society Monographs. New Series, 7.
Oxford Science Publications.
The Clarendon Press, Oxford University Press, New York, 1993.

\bibitem{Shestakov}
Shestakov, I. P.:
Quantization of Poisson superalgebras and the specialty of Jordan superalgebras of Poisson type.
Algebra i Logika 32 (1993), no. 5, 571--584 (1994); translation in 
Algebra and Logic 32 (1993), no. 5, 309--317 (1994).

\bibitem{Weinstein}
Weinstein, Alan:
The local structure of Poisson manifolds.
J. Differential Geom. 18 (1983), no. 3, 523--557.

\bibitem{ZCB}
Zhang, Zerui; Chen, Yuqun; Bokut, Leonid A.:
Some algorithmic problems for Poisson algebras.
J. Algebra 525 (2019), 562--588.

\end{thebibliography}
\end{document}